\documentclass[a4paper,12pt]{article}
\usepackage{amsmath}
\usepackage{amssymb}
\usepackage{amsthm}
\usepackage[english]{babel}
\usepackage{hyperref}
\usepackage{bbm}
\usepackage{algpseudocode}

\newtheorem{theor}{Theorem}
\newtheorem{defi}[theor]{Definition}
\newtheorem{lemma}[theor]{Lemma}

\newtheorem{rem}[theor]{Remarks}

\newcommand{\Z}{\mathbb{Z}}

\newcommand{\F}{\mathbb{F}}
\newcommand{\R}{\mathbb{R}}
\newcommand{\C}{\mathbb{C}}
\newcommand{\1}{\mathbbm{1}}
\newcommand{\dist}{\mathrm{dist}}
\newcommand{\supp}{\mathrm{supp}}

\newcommand{\Median}{\mathrm{Median}}

\newcommand{\Rep}{\mathrm{Rep}}
\newcommand{\Span}{\mathrm{span}}

\renewcommand{\Re}{\mathrm{Re}}
\renewcommand{\Im}{\mathrm{Im}}

\newlength{\dwidehatheight}

\begin{document}

\begin{center}
{\Large 
Explicit universal sampling sets in finite vector spaces}

Lucia Morotti
\end{center}

\begin{abstract}
In this paper we construct explicit sampling sets and present reconstruction algorithms for Fourier signals on finite vector spaces $G$, with $|G|=p^r$ for a suitable prime $p$. The two sampling sets have sizes of order $O(pt^2r^2)$ and $O(pt^2r^3\log(p))$ respectively, where $t$ is the number of large coefficients in the Fourier transform. The algorithms approximate the function up to a small constant of the best possible approximation with $t$ non-zero Fourier coefficients. The fastest of the algorithms has complexity $O(p^2t^2r^3\log(p))$.
\end{abstract}

\section{Introduction}

The problem of compressive sensing originated in the context of Fourier series \cite{MR2230846}. The aim is to reconstruct a linear combination of a small number of complex exponentials from as few samples as possible, when only the number of the exponentials entering the linear combination is known. The additional challenge was to come up with practical and efficient methods for the reconstruction (which by its combinatorial nature is NP-hard, unless extra information is available).
 
 Later on, the compressed sensing problem evolved to include a more general setup. The overall problems and main challenges, however, remained the same; and they concerned mostly the construction of sampling schemes that would allow (and guarantee) efficient reconstruction from as few measurements as possible, and the design of efficient reconstruction algorithms. For the latter, $\ell^1$-minimization turned out to be a popular choice, and the chief technical condition to guarantee success for the reconstruction method was the {\em restricted isometry condition} (see \cite{MR2243152} for the first introduction of these ideas, and \cite{MR3100033} for an in-depth study). However, there still remained the problem of constructing measurement matrices (or, in the Fourier case, sampling sets) for which the RIP was actually provably fulfilled. An important (somewhat partial) answer to this problem was provided by random methods; e.g., in the case of random sampling of Fourier matrices the RIP assumption turns out 
to be true under rather weak assumptions on the number of samples \cite{MR2417886}, at least with high probability. 
 
 However, the case of {\em deterministic sampling sets} with {\em provably} guaranteed reconstruction poses altogether different challenges. Firstly, the verification of properties like RIP is a very complex problem by itself \cite{MR3164973}, hence special care must be taken to allow such estimates. A first successful example for a deterministic construction with the RIP property was presented by DeVore \cite{MR2371999}. For the Fourier setting, in \cite{b1,MR2628828} sampling sets and inversion algorithms were constructed for cyclic groups; an alternative construction of deterministic sampling sets guaranteeing RIP for the cyclic case was developed in \cite{5464880}.  All these construction have a common restriction, commonly known as {\em quadratic bottleneck}: In order to reconstruct linear combinations of $t$ basis vectors, they need $O(t^2)$ samples. A more recent paper by Bourgain and collaborators \cite{MR2817651} managed to improve this to $O(t^{2-\epsilon})$, for a very small $\epsilon>0$, using 
rather 
involved arguments from additive combinatorics. 
 
 This paper considers efficiently sampling Fourier-sparse vectors on finite abelian groups. It can be seen as complementary to \cite{b1,5464880,MR2628828}, with the main difference being that this paper focuses on {\em finite vector spaces} rather than cyclic groups. We develop a general, simple scheme for the design of sampling sets, together with algorithms that allow reconstruction.
 
 Hence the new methods provide an alternative means of explicitly designing {\em universal sampling sets} in a specific family of finite abelian groups $G$, i.e., sampling sets $\Omega \subset G$ that allow the reconstruction of any given linear combination of $t$ characters of $G$ from its restriction on $\Omega$, together with {\em explicit} inversion algorithms, both for the noisy and noise free cases. The groups $G$ we consider are finite vector spaces, and the sampling sets will be written as unions of suitable affine subspaces. The sampling sets actually fulfill the RIP property, which allows one to use the standard methods such as $\ell^1$-minimization. However, the special structure of the sampling set makes the inversion algorithm particularly amenable to the use of a more structured (and potentially faster) reconstruction algorithm, using FFT methods. It should be stressed, though, that our construction is {\em not} able to beat the quadratic bottleneck.

\section{Notation}\label{s1}

Let $p$ be a prime and $r$ a positive integer. When considering their additive groups structure we have that $(\Z/p\Z)^r\cong\F_p^r$. The vector space structure of $\F_p^r$ will enable us to construct the sampling sets needed in the algorithms described in this paper. We will write $\F_p^r$ also when considering only its additive group structure. Also in order to avoid complicated notations we will identify, where needed, elements of $\Z$ with their images in $\F_p$, so that for example 1 could be viewed as either an element of $\Z$ or of $\F_p$.

We will write $H\leq G$ for a subgroup $H$ of $G$. Subsets of $\F_p^r$ are subgroups if and only if they are vector spaces (since $\F_p$ is a field of prime order). So when looking at $\F_p^r$ as a vector space we will write $H\leq \F_p^r$ for a subspace $H$. For any subgroup $H\leq G$ we will also write $\Rep(G/H)$ for a set of representatives of cosets of $H$ in $G$.

In Section \ref{s2} we will also be working with both $\F_p$ and $\F_q$, where $q$ is a power of $p$. Since vector spaces over $\F_q$ can be viewed also as vector spaces over $\F_p$, we will write $\dim_{\F_p}(V)$ and $\dim_{\F_q}(V)$ for the dimension of $V$ as a vector space over $\F_p$ or $\F_q$ respectively. Similarly we will write $\Span_{\F_p}(A)$ and $\Span_{\F_q}(A)$ for the span of $A$ as vector space over $\F_p$ or $\F_q$ respectively.

Let $\widehat{\F_p^r}$ consist of all group homomorphisms $\F_p^r\rightarrow\C$. We have that
\[\widehat{\F_p^r}=\{\chi_{(y_1,\ldots,y_r)}:y_i\in\F_p\}=\{\chi_y:y\in\F_p^r\},\]
where, if $\omega_p=e^{2\pi i/p}$, we define  $\chi_{(y_1,\ldots,y_r)}(x_1,\ldots,x_r):=\omega_p^{x_1y_1+\ldots+x_ry_r}$ for $(x_1,\ldots,x_r)\in\F_p^r$. This is well defined since $\omega_p^p=1$. Also it is easy to check that $\chi_y\chi_z=\chi_{y+z}$ for $y,z\in\F_p^r$.

For any function $f:(\Z/p\Z)^r\rightarrow\C$ let $\widehat{f}:\widehat{\F_p^r}\rightarrow\C$ be its Fourier transform, so that
\[f=\sum_{\chi_y\in\widehat{\F_p^r}}\widehat{f}(\chi_y)\chi_y.\]
Formulas for the Fourier transform for $\F_p^r$ are given by
\[\widehat{f}(\chi_y)=\frac{1}{p^r}\sum_{x\in\F_p^r}f(x)\chi_y(x)^{-1}.\]
Under the bijection $\chi_y\leftrightarrow y$ this corresponds (up to a scalar multiple) to the usual discrete Fourier transform on the grid $\{0,1,\ldots,p-1\}^r$ (under the identification of this grid with $\F_p^r$).

For any function $h:X\rightarrow Y$ and any subset $Z\subseteq X$ let $h|_Z:Z\rightarrow Y$ be the restriction of $h$ to $Z$.

Let now $t$ be a positive integer and assume that $\widehat{f}=\widehat{g}+\epsilon$ with $\widehat{g}$ having support consisting of at most $t$ elements and with $\|\epsilon\|_1$ ``small''. We will present in this paper two algorithms which approximate $\widehat{f}$ by a function $\widehat{f'}$ which, like $\widehat{g}$, also has support consisting of at most $t$ elements. Further $\widehat{f'}$ satisfies
\[\|\widehat{f}-\widehat{f}'\|_1\leq (1+3\sqrt{2})\|\epsilon\|_1,\]
that is $\widehat{f'}$ is close to the best possible such approximation of $\widehat{f}$. In particular if $\|\epsilon\|_1=0$ then $\widehat{f'}=\widehat{f}$, so that the algorithms reconstruct $\widehat{f}$ (and then also $f$) in this case.

The algorithms presented here are based on similar algorithms presented in \cite{b1,MR2628828} for the group  of unit length elements of $\C$. 
 As in \cite{b1,MR2628828} sampling sets will be obtained by taking unions of cosets of subgroups. In \cite{b1,MR2628828} subgroups of distinct prime orders are considered. This cannot be applied here since $\F_p^r$ is a $p$-group. However, using the vector space structure of $\F_p^r$, we can construct many small subgroups with certain ``orthogonality'' properties (as will be seen in the next section), which will be used to construct the sampling sets presented here.

\section{Complexities of sets and algorithms}

Before describing the sampling sets and the corresponding algorithms, we present their size/running time complexities and compare them to those presented in other papers.

\vspace{6pt}
\noindent
\begin{tabular}{|l|l|l|l|}
\hline
&(Size of) group&Sampling set&Time complexity\\
\hline
Theorem \ref{t2}&$\F_p^r$&$O(pt^2r^2)$&$O(p^rt^2r^2)$\\
\hline
Theorem \ref{t9'}&$\F_p^r$&$O(pt^2r^3\log(p))$&$O(p^2t^2r^3\log(p))$\\
\hline
Algorithm 1 of \cite{b1}&$\Z/n\Z$&$O(\frac{t^2\log(n)^2}{\log(t\log(n))})$&$O(\frac{t^2\log(n)^3}{\log(t\log(n))})$\\
\hline
Section 2 of \cite{MR2817651}&$\Z/p\Z$&$O(t^{2-\epsilon})$&$O(pt^{2-\epsilon})$\\
\hline
Section 3 of \cite{MR2371999}&n&$O(\frac{t^2\log(n/t)^2}{\log(t)^2})$&$O(n\frac{t^2\log(n/t)^2}{\log(t)^2})$\\
\hline
Section 2 of \cite{5464880}&$\Z/p\Z$&$\Omega(t^2)$&$\Omega(pt^2)$\\
\hline
Algorithm 1 of \cite{MR2628828}&$\Z/n\Z$&$O(t^2\log(n)^2)$&$O(nt\log(n)^2)$\\
\hline
Algorithm 2 of \cite{MR2628828}&$\Z/n\Z$&$O(t^2\log(n)^4)$&$O(t^2\log(n)^4)$\\
\hline
\end{tabular}

\vspace{6pt}

As can be seen from the table, the second algorithm presented here needs more sampling points then the first one (it needs about $r\log(p)=\log(|\F_p^r|)$ times as many sampling points). However the first algorithm has a factor $p^r=|\F_p^r|$ in its running time, while the second algorithm is much faster.

The complexities for the construction of the sampling sets used in Theorem \ref{t2} and \ref{t9'} are given as follows.

\vspace{6pt}
\noindent
\begin{tabular}{|l|l|}
\hline
&Time complexity for the construction of the sampling set\\
\hline
Theorem \ref{t2}&$O(pt^2r^5\log(p)^2)$\\
\hline
Theorem \ref{t9'}&$O(pt^2r^6\log(p)^3)$\\
\hline
\end{tabular}

\section{Construction of the sampling sets}\label{s2}

The sets constructed here are unions of (shifted) subspaces of $\F_p^r$. We begin with a lemma that will allow us to find certain families of subspaces which will be of basic importance for the construction of our sampling sets.

\begin{lemma}\label{l1}
Let $p$ be a prime. Also let $1\leq h\leq r$ be integer and define $q:=p^h$ and $s:=\lceil r/h\rceil$. Let $\pi:\F_q^s\rightarrow\F_p^r$ be a surjective homomorphism and fix $a_1,\ldots,a_n\in\F_q^s$. For $1\leq i\leq n$ integer choose $H_i\leq\F_p^r$ of dimension $h$ containing $\pi(\Span_{\F_q}\{ a_i\})$. Fix a positive integer $m\leq n$. If
\[\Span_{\F_q}\{ a_{i_1},\ldots,a_{i_m}\}=\F_q^s\]
for every $1\leq i_1<\ldots<i_m\leq n$ then
\[\Span_{\F_p}\{ H_{i_1},\ldots,H_{i_m}\}=\F_p^r\]
for every $1\leq i_1<\ldots<i_m\leq n$.
\end{lemma}

In the lemma $\{a_{i_1},\ldots,a_{i_m}\}$ is a subset of $\{a_1,\ldots,a_n\}$ with $m$ elements. Also, if $z$ is a primitive element of $\F_q$, so that elements of $\F_q$ can be written as $b_0+b_1z+\ldots+b_{h-1}z^{h-1}$ with $b_i\in\F_p$, we can take $\pi:\F_q^s\rightarrow\F_p^r$ to be the composition of $\pi_1:\F_q^s\rightarrow\F_p^{hs}$ and $\pi_2:\F_p^{hs}\rightarrow\F_p^r$ given as follows
\begin{align*}
\pi_1(\sum b_i^{(1)}z^i,\ldots,\sum b_i^{(s)}z^i)&:=(b_0^{(1)},\ldots,b_{h-1}^{(1)},\ldots,b_0^{(s)},\ldots,b_{h-1}^{(s)}),\\
\pi_2(c_1,\ldots,c_{hs})&:=(c_1,\ldots,c_r).
\end{align*}
If $\pi(\Span_{\F_q}\{ a_i\})$ has dimension $h$ as vector space over $\F_p$, then we have to take $H_i=\pi(\Span_{\F_q}\{ a_i\})$. If $\pi(\Span_{\F_q}\{ a_i\})$ has dimension less than $h$ as vector space over $\F_p$, then we can choose $H_i=\Span_{\F_p}\{\pi(\Span_{\F_q}\{ a_i\}),e_1,\ldots,e_j\}$ for a certain $1\leq j\leq r$ (here $e_k$ is the $k$-th standard basis element of $\F_p^r$).

\begin{proof}
Notice that a surjective homomorphism $\pi:\F_q^s\rightarrow\F_p^r$ always exists since, as $\F_p$-vector spaces,
\[\F_p^r\leq\F_p^{h\times\lceil r/h\rceil}\cong\F_q^s.\]
In particular $\dim_{\F_p}(\F_p^r)\leq\dim_{\F_p}(\F_q^s)$. Also we can construct $H_i$, as
\[\dim_{\F_p}(\pi(\Span_{\F_q}\{ a_i\}))\leq\dim_{\F_p}(\Span_{\F_q}\{ a_i\})=h\dim_{\F_q}(\Span_{\F_q}\{ a_i\})\leq h.\]

If $\Span_{\F_q}\{ a_{i_1},\ldots,a_{i_m}\}=\F_q^s$ then
\[\F_p^r\geq\Span_{\F_p}\{ H_{i_1},\ldots,H_{i_m}\}\geq\pi(\Span_{\F_q} \{a_{i_1},\ldots,a_{i_m}\})=\pi(\F_q^s)=\F_p^r\]
and so the lemma follows.
\end{proof}

If some subspaces $H_1,\ldots,H_n$ of $\F_p^r$ satisfy $\Span_{\F_p}\{ H_{i_1},\ldots,H_{i_m}\}=\F_p^r$ for every $1\leq i_1<\ldots<i_m\leq n$, we say that $H_1,\ldots,H_n$ are \emph{$m$-generating}. Similarly, if $a_1,\ldots,a_n\in\F_q^s$ satisfy $\Span_{\F_q}\{ a_{i_1},\ldots,a_{i_m}\}=\F_q^s$ for every $1\leq i_1<\ldots<i_m\leq n$, we say that $\{a_1,\ldots,a_n\}$ is \emph{$m$-generating}. For $m=s$  we have that $\{a_1,\ldots,a_n\}\subseteq \F_q^s$ is $s$-generating if and only if any $s$ of its elements are linearly independent. So a subset of $\F_q^s$ is $s$-generating if and only if it has full spark.

It is not always possible to construct a $m$-generating subset of $\F_q^s$ of size $n$, for any $n$. We will in the next lemma show though that $\F_q^s$ has a $s$-generating multiset for $\F_q^s$ of size $q+1$ (we will consider multisets and not sets in order to cover also the case $s=1$) and construct one such multiset explicitly.

\begin{lemma}\label{t51}
Let
\[A:=\{(1,x,\ldots,x^{s-1}):x\in\F_q\}\cup\{(0,\ldots,0,1)\in\F_q^s\}.\]
Then $A\subseteq\F_q^s$ is $s$-generating and $|A|=q+1$.
\end{lemma}

\begin{proof}
It's clear that $|A|=q+1$. For the proof of $A$ being $s$-generating see the beginning of Chapter 11 \S 5 of \cite{ms}. This can also be seen by noticing that the matrices obtained from any $s$ distinct elements of $A$ all have one of the following forms (up to exchanging rows):
\[\left(\begin{array}{cccc}
1&x_1&\cdots&x_1^{s-1}\\
\vdots&\vdots&&\vdots\\
1&x_s&\cdots&x_s^{s-1}
\end{array}\right)\hspace{24pt}\mbox{or}\hspace{24pt}\left(\begin{array}{ccccc}
1&x_1&\cdots&x_1^{s-2}&x_1^{s-1}\\
\vdots&\vdots&&\vdots\\
1&x_{s-1}&\cdots&x_{s-1}^{s-2}&x_{s-1}^{s-1}\\
0&0&\cdots&0&1
\end{array}\right)\]
with distinct $x_j$ and so such matrices are invertible.
\end{proof}

In order to construct one of the two sampling sets needed in the algorithms, we still need a stable sampling set for 1-sparse Fourier signals on $\F_p$. 
%
%
We will show in the next lemma how such a set can be found. In the same lemma we will also give an algorithm which, applied to the function $f(x)=\sum_{y=0}^{p-1} a_y\omega_p^{yx}$, under certain assumptions on $f$, returns $y'$ with $|a_{y'}|$ maximal. In order to state the theorem we need the following definition.


\begin{defi}
For $d>0$ and $x\in\R$ let $|x|_d:=\dist(x,d\Z)$ be the minimum distance of $x$ from an integer multiple of $d$.
\end{defi}

\begin{lemma}\label{t11}
Let $K:=\{0\}\cup\{2^i:0\leq i\leq k\mbox{ and }i\in\Z\}\subseteq\F_p$, where $k\in\Z$ is minimal such that $2^k\geq p/3$.  Let $f(x)=\sum_{y=0}^{p-1} a_y\chi_y(x)$. If there exists $y'$ such that $|a_{y'}|>2\sum_{y\not=y'}|a_y|$, then the following algorithm returns $y'$.

\begin{algorithmic}
\State Set $b(2^l):=\arg(f(2^l)/f(0))$, for $0\leq l\leq k$ (if $f(2^l)$ or $f(0)$ are zero set $b(2^l):=0$).

\State Set $e:=(0,\ldots,0)$ ($Length(e)=p$).

\For {$l$ from $0$ to $k$}

\For {$j$ from $0$ to $p-1$}

\If {$e(j+1)=0$ and $|p\,b(2^l)/(2\pi)-2^lj|_p\geq p/6$}

\State $e(j+1):=1$

\EndIf

\EndFor

\EndFor

\noindent\Return the index $j$ such that $e(j+1)=0$. If such a $j$ does not exist, return 0.
\end{algorithmic}
\end{lemma}

\begin{proof}
If no $y'$ exists then there is nothing to prove, so we will now assume that there exists $y'$ with $|a_{y'}|>2\sum_{y\not=y'}|a_y|$. Then, for $x\in\F_p$, we have that
\[\left|\sum_{y\not=y'} a_y\chi_y(x)\right|=\left|\sum_{y\not=y'} a_y\omega_p^{xy}\right|\leq \sum_{y\not=y'}\left| a_y\omega_p^{xy}\right|=\sum_{y\not=y'}\left| a_y\right|<|a_{y'}|/2\]
and so
\[|f(x)|=\left|\sum_{y=0}^{p-1} a_y\chi_y(x)\right|=\left|\sum_{y=0}^{p-1} a_y\omega_p^{xy}\right|\geq \left|a_{y'}\omega_p^{xy'}\right|-\left|\sum_{y\not=y'} a_y\omega_p^{xy}\right|>|a_{y'}|/2\geq 0.\]
So $f(x)\not= 0$ for every $x\in\F_p$. In particular $b(2^i)=\arg(f(2^l)/f(0))$.

For $z\in\C$ with $|z|<1/2$ we have that $|\arg(1+z)|_{2\pi}<\pi/6$. It then follows that
\[|\arg(f(x))-\arg(a_{y'}\omega_p^{xy'})|_{2\pi}=|\arg(f(x)/(a_{y'}\omega_p^{xy'}))|_{2\pi}< \pi/6\]
and so
\begin{align*}
|\arg(f(x)/f(0))-\arg(\omega_p^{xy'})|_{2\pi}&=|\arg(f(x))-\arg(f(0))\\
&\hspace{24pt}-\arg(a_{y'}\omega_p^{xy'})+\arg(a_{y'})|_{2\pi}\\
&\leq|\arg(f(x))-\arg(a_{y'}\omega_p^{xy'})|_{2\pi}\\
&\hspace{24pt}+|\arg(f(0))-\arg(a_{y'})|_{2\pi}\\
&< \pi/3
\end{align*}
for every $x\in\F_p$. Then
\begin{align*}
|p\,b(2^l)/(2\pi)-2^ly'|_p&=|p/(2\pi)(\arg(f(2^l)/f(0))-\arg(\omega_p^{2^ly'}))|_p\\
&=p/(2\pi) |\arg(f(2^l)/f(0))-\arg(\omega_p^{2^ly'})|_{2\pi}\\
&< p/6
\end{align*}
for $0\leq l\leq k$, in particular $e(y'+1)=0$.

We will now show that if $y\not= y'$ then $e(y+1)=1$, which will then prove the lemma. In order to do this we will show that if $j\not=0$ then there exists $l$ such that $0\leq l\leq k$ and $|2^lj|_p\geq p/3$, which also proves that if $y\not= y'$ then there exists $l$, $0\leq l\leq k$, with $|2^l(y'-y)|_p\geq p/3$. Hence, for the same $l$ we have that
\[|p\,b(2^l)/(2\pi)-2^ly|_p\geq|2^l(y'-y)|_p-|p\,b(2^l)/(2\pi)-2^ly'|_p\geq p/3-p/6=p/6.\]

Assume that $j\not\equiv 0 \mod p$ and that $|2^lj|_p<p/3$ for all $0\leq l\leq k$. Then, up to a multiple of $p$, we also have that $j\in\pm\{1,\ldots,\lceil{p/3}\rceil-1\}$ (by considering the case $l=0$). As $|2^lj|_p=|2^l(-j)|_p$ we can assume that $j\in\{1,\ldots,\lceil{p/3}\rceil-1\}$. Let $l$ be minimal such that $2^lj\geq p/3$. As $1\leq j<p/3$ we have that $1\leq l\leq k$ by definition of $k$. As $1\leq 2^{(l-1)}j<p/3$ it follows that $p/3\leq 2^lj<2p/3$ and so $|2^lj|_p\geq p/3$. Since $1\leq l\leq k$ this gives a contradiction and so the lemma is proved.
\end{proof}

We will now construct the sampling sets which will be used in the algorithms presented in the next section.

\begin{defi}\label{d1}
Let $m\geq 1$ and $n:=4t(m-1)+1$ and let $H_1,\ldots,H_n\leq \F_p^r$ be $m$-generating and all of dimension $h$ with $1\leq h\leq r$. Also let $K\subseteq\F_p$ be as in Lemma \ref{t11}. For $1\leq i\leq n$ choose $x_{1,i},\ldots,x_{r-h,i}\in\F_p^r$ such that $\Span_{\F_p}\{ H_i,x_{1,i},\ldots,x_{r-h,i}\}=\F_p^r$. The wanted sets are
\begin{align*}
\Gamma_1:=&\cup_i H_i,\\
\Gamma_2:=&\cup_i (H_i+\{kx_{i,j}:k\in K,\,\,1\leq j\leq r-h\}),
\end{align*}
where for sets $A$ and $B$ we define $A+B:=\{a+b:a\in A,\, b\in B\}$.
\end{defi}


\begin{rem}\label{r1}
From Lemmas \ref{l1} and \ref{t51}, $m$-generating $H_1,\ldots,H_n$ exist if $m=\lceil r/h\rceil$ and $4t(\lceil r/h\rceil-1)\leq p^h$. Also $x_{1,i},\ldots,x_{r-h,i}$ exists for $1\leq i\leq n$ as $H_i$ is of dimension $h$.

It can be easily checked from the definitions that
\begin{align*}
|\Gamma_1|&\leq np^h,\\
|\Gamma_2|&\leq np^h|K|(r-h).
\end{align*}
\end{rem}

We will now bound $np^h$ and $|K|$.

\begin{theor}\label{c20}
Let $h\in\Z$ be minimal such that $h\geq 1$ and $4t(\lceil r/h\rceil-1)\leq p^h$ and let $K$ as in Lemma \ref{t11}. Also let $n:=4t(\lceil r/h\rceil-1)+1$. Then $np^h<16pt^2r^2$ and $|K|\leq 2+\log_2(p)$.
\end{theor}

\begin{proof}
Notice that $h\leq r$, since $4t(\lceil r/r\rceil-1)=0\leq p^r$

If $h=1$ we have that $np^h\leq (4t(r-1)+1)p\leq 4ptr<16pt^2r^2$.

If $h>1$ then $p^{h-1}< 4t(\lceil r/(h-1)\rceil -1)<4tr$ and so
\[np^h\leq pp^{h-1}\left(4t\left(\left\lceil\frac{r}{h}\right\rceil-1\right)+1\right)<p\cdot4tr\cdot\left(\frac{4tr}{h}+1\right)\leq16pt^2r^2.\]

When looking at $|K|$, we have that $|K|=2+k$, where $k\in\Z$ is minimal such that $2^k\geq p/3$. So $2^k<2p/3$ and then $k<\log_2(2p/3)<\log_2(p)$.
\end{proof}

From Remark \ref{r1} and Theorem \ref{c20} we obtain bounds on $|\Gamma_1|$ and $|\Gamma_2|$.

\begin{rem}\label{r2}
From Remark \ref{r1} and Theorem \ref{c20} we obtain that $\Gamma_1$ and $\Gamma_2$ can be chosen with
\begin{align*}
|\Gamma_1|&\leq 16pt^2r^2,\\
|\Gamma_2|&\leq 16pt^2r^3(2+\log_2(p)).
\end{align*}
\end{rem}

We still need to prove that the sampling sets constructed in Definition \ref{d1} are actually sampling sets. This will be done in the next section by proving that the reconstruction algorithms which will be presented work. Since the reconstruction algorithms need values $f(x)$ with $x\in\Gamma_1$ or $x\in\Gamma_2$ respectively, this will also prove that the given sets are sampling sets. An alternative way would be to prove that $\Gamma_1$ and $\Gamma_2$ satisfy the RIP property (see Section \ref{s3} for remarks about it). 

\section{Preliminaries to the reconstruction algorithms}\label{s5}

We will now prove a lemma which will play a crucial role in the two reconstruction algorithms we will present in Section \ref{s4}. For any $H\leq \F_p^r$ and $L\leq \widehat{\F_p^r}$ let $H^\perp:=\{\chi_y\in\widehat{\F_p^r}:\chi_y(h)=1\,\,\,\,\forall h\in H\}$ and $L^\perp:=\{x\in\F_p^r:\chi_y(x)=1\,\,\,\,\forall \chi_y\in L\}$ be the annihilators of $H$ and $L$ respectively.

\begin{lemma}\label{l8}
Let $H_1,\ldots,H_n$ be $m$-generating with $n=4t(m-1)+1$, $f=\sum_{\chi_y\in\widehat{\F_p^r}}a_{\chi_y}\chi_y$ be a $t$-sparse Fourier function. Write $\widehat{f}=\widehat{g}+\epsilon$ with the support of $\widehat{g}$ having at most $t$ elements. Then for every $\chi_z\in\widehat{\F_p^r}$ we have that
\[\left|\left\{1\leq i\leq n:\sum_{\chi_y\in\chi_z H_i^\perp\setminus\{\chi_z\}}|a_{\chi_y}|\leq\frac{\|\epsilon\|_1}{t}\right\}\right|\geq 2t(m -1)+1.\]
\end{lemma}

\begin{proof}
First we will show that the matrix $M$ is $(n,m-1)$-coherent, and then we will apply (a variant of) Lemma 2 of \cite{b1} to conclude the proof of the lemma.

Let $M$ be the matrix with rows labeled by
\[A:=\{\chi_y H^\perp_i:1\leq i\leq n,\,\,\chi_y\in\Rep(\widehat{\F_p^r}/H^\perp_i)\},\]
columns labeled by elements of $\widehat{\F_p^r}$ and $M_{\chi_y H^\perp_i,\chi_z}=1$ if $\chi_z\in\chi_y H^\perp_i$ or $M_{\chi_y H^\perp_i,\chi_z}=0$ otherwise. Then each column of $M$ contains exactly $n$ entries equal to 1 and
\[\sum_{a\in A}M_{a,\chi_z}M_{a,\chi_w}=\sum_{i=1}^nM_{\chi_z H^\perp_i,\chi_w}=|\{1\leq i\leq n:(\chi_z)^{-1}\chi_w\in H^\perp_i\}|\]
for $\chi_z,\chi_w\in\widehat{\F_p^r}$. Let $\langle\chi_x\rangle\subseteq\widehat{\F_p^r}$ be the subgroup generated by $\chi_x$. Since the $H_i$ are $m$-generating, so that no more than $m-1$ of them can be contained in a fixed proper subspace of $\F_p^r$, and since $\langle\chi_x\rangle^\perp\lneq \F_p^r$ for $1\not=\chi_x\in \widehat{\F_p^r}$ (that is $x\not=0\in\F_p^r$), we have that
\begin{align*}
\sum_{a\in A}M_{a,\chi_z}M_{a,\chi_w}&=|\{1\leq i\leq n:(\chi_z)^{-1}\chi_w\in H^\perp_i\}|\\
&=|\{1\leq i\leq n:\chi_{w-z}\in H^\perp_i\}|\\
&=|\{1\leq i\leq n:\langle\chi_{w-z}\rangle\leq H^\perp_i\}|\\
&=|\{1\leq i\leq n:H_i\leq\langle\chi_{w-z}\rangle^\perp\}|\\
&\leq m-1
\end{align*}
for $\chi_z\not=\chi_w$ (that is $x\not=w$ as elements of $\F_p^r$).

Assume first that $\epsilon\not=0$. If $h<r$ the lemma then follows from Lemma 2 of \cite{b1} with (using the notation of Lemma 2 of \cite{b1}) $k=t$ (we have that $t<n$ by assumption), $\epsilon'=1$ and $c=4$. If $h=r$ then $H_i^\perp=\{1\}$ for every $1\leq i\leq n$ and so
\[\sum_{\chi_y\in\chi_z H_i^\perp\setminus\{\chi_z\}}|a_{\chi_y}|=0\leq\frac{\|\epsilon\|_1}{t} \]
for every $\chi_z\in\widehat{\F_p^r}$ and $1\leq i\leq n$. In particular the lemma holds also in this case.

Assume now that $\epsilon=0$. Then
\[f=\sum_{j=1}^ta_{\chi_{y_j}}\chi_{y_j}\]
for some $\chi_{y_j}\in\widehat{F_p^r}$. In particular
\[\sum_{\chi_y\in\chi_z H_i^\perp\setminus\{\chi_z\}}|a_{\chi_y}|=\sum_{{1\leq j\leq t:\chi_{y_j}\not=\chi_z,}\atop{\chi_{y_j}\in\chi_z H_i^\perp}}|a_{\chi_{y_j}}|.\]
Since $\chi_{y_j}\in\chi_z H_i^\perp$ if and only if $(\chi_z)^{-1}\chi_{y_j}\in H_i^\perp$ and, for each $\chi_{y_j}\not=\chi_z$, there exists at most $m-1$ such $i$, we have that
\[\sum_{\chi_y\in\chi_z H_i^\perp\setminus\{\chi_z\}}|a_{\chi_y}|=0\]
for at least $n-t(m-1)\geq 2t(m-1)+1$ distinct values of $i$. In particular the lemma holds also in this case.
\end{proof}

Let $H$ be any subgroup of $\F_p^r$. Then $\widehat{\F_p^r}/H^\perp\cong\widehat{H}$ through $\chi_y H^\perp\mapsto(h\mapsto\chi_y(h))$. For any function $f=\sum_{\chi_y\in\widehat{\F_p^r}}a_{\chi_y} \chi_y$ on $\F_p^r$ we have that
\begin{equation}\label{eq1}
\widehat{f|_H}(\chi_y H^\perp)=\sum_{\chi_z\in\chi_y H^\perp}a_{\chi_z}=\sum_{\chi_z\in\chi_y H^\perp}\widehat{f}(\chi_z)
\end{equation}
since
\[f|_H=\sum_{\chi_y\in\widehat{\F_p^r}}a_{\chi_y} \chi_y|_H=\sum_{\chi_y H^\perp\in \widehat{\F_p^r}/H^\perp}\chi_y|_H\sum_{\chi_z\in\chi_y H^\perp}a_{\chi_z}.\]
This identification will be used in the proof of the algorithms which we will present in the next section.

 \section{Reconstruction algorithms}\label{s4}
  
We are now ready to present the reconstruction algorithms. Through all of this section let $f$, $g$ and $\epsilon$ be as in Section \ref{s1} and $\Gamma_j$, $H_1,\dots,H_n$, $K$, $m$, $n$, $h$ and $x_{l,i}$ as in Definition \ref{d1}. The first algorithm we present reconstructs or approximates $\widehat{f}$ from $f|_{\Gamma_1}$.

\begin{theor}\label{t2}
For $1\leq j\leq n$ let $F_j$ be the Fourier transform for $H_j$ and define $c_j:=F_j(f|_{H_j})$. Then the following algorithm returns $\widehat{f}'$ with $|\supp(\widehat{f}')|\leq t$ and $\|\widehat{f}-\widehat{f}'\|_1\leq (1+3\sqrt{2})\|\epsilon\|_1$.

\vspace{12pt}
\begin{algorithmic}
\State Set $\widehat{f}':=(0,\ldots,0)$, $\widehat{f}'':=(0,\ldots,0)$ and $Y:=(0,\ldots,0)$

\For {$j$ from $1$ to $n$}


\For {$2t-1$ values of $\chi_y H_j^\perp$ such that $|c_j(\chi_y H_j^\perp)|$ is largest}

\For {$\chi_z\in\chi_y H_j^\perp$}

\State $Y(\chi_z):=Y(\chi_z)+1$

\EndFor

\EndFor


\EndFor

\For {$\chi_y\in\widehat{\F_p^r}$}

\If {$Y(\chi_y)>2t(m -1)$}

\State $X:=()$

\For {$j$ from $1$ to $n$}

\State Append $c_j(\chi_y H_j^\perp)$ to $X$

\EndFor

\State $\widehat{f}''(\chi_y):=\Median(\{\Re(x):x\in X\})+i\Median(\{\Im(x):x\in X\})$

\EndIf

\EndFor

\For {$t$ values of $\chi_y$ for which $|\widehat{f}''(\chi_y)|$ is largest}

\State $\widehat{f}'(\chi_y):=\widehat{f}''(\chi_y)$

\EndFor

\noindent \Return $\widehat{f}'$.
\end{algorithmic}
\end{theor}

In the algorithm $\widehat{f}'$, $\widehat{f}''$ and $Y$ are labeled by elements of $\widehat{\F_p^r}$. The Fourier transform for $H_j$ can be defined similarly to that of $\F_p^r$ (notice that $H_j\cong\F_p^h$ by definition). Under the identification at the end of Section \ref{s5} giving $\widehat{H_j}\cong\widehat{F_p^r}/H^\perp$ we can also define $F_j$ as a function $F_j:\widehat{F_p^r}/H^\perp\rightarrow \C$.

\begin{proof}

We will first show that if $|a_{\chi_z}|> 2\|\epsilon\|_1/t$ then $Y(\chi_z)> 2t(m-1)$ and next prove that if $Y(\chi_w)> 2t(m-1)$ then $|\widehat{f}''(\chi_w)-a_{\chi_w}|\leq\sqrt{2}\|\epsilon\|_1/t$. Using these results we will then prove the theorem.

By definition of $\epsilon$, for $1\leq j\leq n$ we have that
\begin{equation}\label{eq2}
|\{\chi_y H_j^\perp\in\widehat{\F_p^r}/H_j^\perp:|c_j(\chi_y H_j^\perp)|>\|\epsilon\|_1/t\}|\leq 2t-1,
\end{equation}
since, if $g=\sum_{l=1}^t a_{\chi_{y_l}}\chi_{y_l}$ (where some of the coefficients might be 0) and $\{\chi_{y_l} H_j^\perp\}:=\{\chi_{y_1} H_j^\perp,\ldots,\chi_{y_t} H_j^\perp\}$, then, from Equation \eqref{eq1},
\begin{align*}
\sum_{\chi_y H_j^\perp\in(\widehat{\F_p^r}/H^\perp)\setminus\{\chi_{y_l} H_j^\perp\}}|c_j(\chi_y H_j^\perp)|&=\sum_{\chi_y H_j^\perp\in(\widehat{\F_p^r}/H_j^\perp)\setminus\{\chi_{y_l} H_j^\perp\}}\left|\sum_{\chi_b\in\chi_y H_j^\perp} a_{\chi_b}\right|\\
&\leq\sum_{\chi_y H_j^\perp\in(\widehat{\F_p^r}/H_j^\perp)\setminus\{\chi_{y_l} H_j^\perp\}}\sum_{\chi_b\in\chi_y H_j^\perp} |a_{\chi_b}|\\
&\leq\sum_{\chi_y\in\widehat{\F_p^r}\setminus\{\chi_{y_1},\ldots,\chi_{y_t}\}}|a_{\chi_y}|\\
&\leq \|\epsilon\|_1
\end{align*}
and then in particular there are less than $t$ elements $\chi_y H_j^\perp\in(\widehat{\F_p^r}/H_j^\perp)\setminus\{\chi_{y_l} H_j^\perp\}$ such that $|c_j(\chi_y H_j^\perp)|>\|\epsilon\|_1/t$. In particular there are at most $2t-1$ elements $\chi_y H_j^\perp\in(\widehat{\F_p^r}/H_j^\perp)$ with $|c_j(\chi_y H_j^\perp)|>\|\epsilon\|_1/t$.

If $|a_{\chi_z}|> 2\|\epsilon\|_1/t$ then
\[|\{j:|c_j(\chi_z H^\perp)|>\|\epsilon\|_1/t\}|>2t(m-1)\]
by Lemma \ref{l8} and Equation \eqref{eq1} as then for at least $2t(m-1)$ values $j$ we have
\[|c_j(\chi_z H_j^\perp)|=\left|\sum_{\chi_y\in\chi_z H_j^\perp}a_{\chi_y}\right|\geq|a_{\chi_z}|-\left|\sum_{\chi_y\in\chi_z H_j^\perp\setminus\{\chi_z\}}a_{\chi_y}\right|>\|\epsilon\|_1/t.\]
So $Y(\chi_z)> 2t(m-1)$ in this case.

Again by Lemma \ref{l8} and Equation \eqref{eq1}, we have that whenever $Y(\chi_w)> 2t(m-1)$ then $|\widehat{f}''(\chi_w)-a_{\chi''}|\leq\sqrt{2}\|\epsilon\|_1/t$, since in this case $\widehat{f}''(\chi_w)=\Median(\{\Re(x):x\in X\})+i\Median(\{\Im(x):x\in X\})$ and
\begin{align*}
|\Median(\{\Re(x):x\in X\})-\Re(a_{\chi_w})|&\leq\|\epsilon\|_1/t,\\
|\Median(\{\Im(x):x\in X\})-\Im(a_{\chi_w})|&\leq\|\epsilon\|_1/t.
\end{align*}

We will now prove that $\|\widehat{f}-\widehat{f}'\|_1\leq (1+3\sqrt{2})\|\epsilon\|_1$, which will prove the theorem, since by definition $|\supp(\widehat{f}')|\leq t$. To do this let $T:=\{\chi_{y_1},\ldots,\chi_{y_t}\}$ and $T':=\supp(\widehat{f}')$. We can write $T\setminus T'=T_1\cup T_2$, where
\begin{align*}
T_1&=\{\chi_y\in T\setminus T':|a_{\chi_y}|\leq 2\|\epsilon\|_1/t\},\\
T_2&=\{\chi_y\in T\setminus T':|a_{\chi_y}|> 2\|\epsilon\|_1/t\}.
\end{align*}

Notice that $T',T_2\subseteq \supp(\widehat{f}'')$. For $\chi_b\in T_2$ and $\chi_c\in T'\setminus T$ we have
\[|a_{\chi_c}|+\sqrt{2}\|\epsilon\|_1/t\geq |\widehat{f}''(\chi_c)|\geq|\widehat{f}''(\chi_b)|\geq|a_{\chi_b}|-\sqrt{2}\|\epsilon\|_1/t\]
and so $|a_{\chi_b}|\leq |a_{\chi_c}|+2\sqrt{2}\|\epsilon\|_1/t$.

Also for $\chi_b\in T_2$ we have by the previous part that $\widehat{f}''(\chi_b)\not=0$. From the definition of $\widehat{f}'$ if $|T_2|>0$ then $|T'|=t$. In this last case, from $T_2\subseteq T\setminus T'$ and $|T|=|T'|$, it follows that $|T'\setminus T|\geq |T_2|$ (this last inequality holds also if $|T_2|=0$).

In particular
\begin{align*}
\sum_{\chi_b\in T\setminus T'}|a_{\chi_b}|&=\sum_{\chi_b\in T_1}|a_{\chi_b}|+\sum_{\chi_b\in T_2}|a_{\chi_b}|\\
&\leq |T_1|2\frac{\|\epsilon\|_1}{t}+|T_2|2\sqrt{2}\frac{\|\epsilon\|_1}{t}+\sum_{\chi_c\in T'\setminus T}|a_{\chi_c}|\\
&\leq 2\sqrt{2}\epsilon+\sum_{\chi_c\in T'\setminus T}|a_{\chi_c}|
\end{align*}
and then
\begin{align*}
\left\|\widehat{f}-\widehat{f}'\right\|_1\!\!=&\left\|\widehat{f}|_{T'}-\widehat{f}'|_{T'}\right\|_1\!\!+\!\left\|\widehat{f}|_{T\setminus T'}-\widehat{f}'|_{T\setminus T'}\right\|_1\!\!+\!\left\|\widehat{f}|_{\widehat{\F_p^r}\setminus (T\cup T')}-\widehat{f}'|_{\widehat{\F_p^r}\setminus (T\cup T')}\right\|_1\\
=&\left\|\widehat{f}|_{T'}-\widehat{f}''|_{T'}\right\|_1\!\!+\!\sum_{\chi_b\in T\setminus T'}|a_{\chi_b}|\!+\!\sum_{\chi_y\in \widehat{\F_p^r}\setminus(T\cup T')}|a_{\chi_y}|\\
\leq &|T'|\sqrt{2}\frac{\|\epsilon\|_1}{t}+2\sqrt{2}\|\epsilon\|_1+\sum_{\chi_c\in T'\setminus T}|a_{\chi_c}|+\!\sum_{\chi_y\in \widehat{\F_p^r}\setminus(T\cup T')}|a_{\chi_y}|\\
\leq &3\sqrt{2}\|\epsilon\|_1+\!\sum_{\chi_y\not\in T}|a_{\chi_y}|\\
\leq &(1+3\sqrt{2})\|\epsilon\|_1
\end{align*}
and so the theorem is proved.
\end{proof}

Before presenting the second algorithm we need the following lemma.

\begin{lemma}\label{l7}
Let $H$ be any subgroup of $\F_p^r$, $x\in \F_p^r$ and $f=\sum_{\chi_w\in\widehat{\F_p^r}} a_{\chi_w} \chi_w$. If $\overline{f}$ is a function of $H$ with $\overline{f}(y):=f(x+y)$, $y\in H$, then
\[\widehat{\overline{f}}(\chi_w H^\perp)=\sum_{\chi_z\in\chi_w H^\perp}a_{\chi_z}\chi_z(x).\]
\end{lemma}

\begin{proof}
We have that
\begin{align*}
\overline{f}(y)&=\sum_{\chi_w\in\widehat{\F_p^r}} a_{\chi_w}\chi_w(x+y)\\
&=\sum_{\chi_w\in\widehat{\F_p^r}} a_{\chi_w}\chi_w(x)\chi_w(y)\\
&=\sum_{\chi_w H^\perp\in\widehat{\F_p^r}/H^\perp}\sum_{\chi_z\in\chi_w H^\perp}a_{\chi_z}\chi_z(x)\chi_z(y)\\
&=\sum_{\chi_w H^\perp\in\widehat{\F_p^r}/H^\perp}\chi_w(y)\sum_{\chi_z\in\chi_w H^\perp}a_{\chi_z}\chi_z(x)
\end{align*}
from which the lemma follows.
\end{proof}

We will now present the second reconstruction theorem. Here we use $f|_{\Gamma_2}$ in order to construct $\widehat{f}'$.

\begin{theor}\label{t9'}
For $1\leq j\leq n$ and $x\in \F_p^r$ let $g_{j,x}(y)=f(x+y)$, $y\in H_j$, and $c_{j,x}=F_j(g_{j,x})$, where $F_j$ is the Fourier transform for $H_j$. The following algorithm returns $\widehat{f}'$ such that $|\supp(\widehat{f}')|\leq t$ and $\|\widehat{f}-\widehat{f}'\|_1\leq (1+3\sqrt{2})\|\epsilon\|_1$.

\begin{algorithmic}
\State Set $\widehat{f}':=(0,\ldots,0)$, $\widehat{f}'':=(0,\ldots,0)$, $Y:=(0,\ldots,0)$ and $Z:=()$

\For {$j$ from $1$ to $n$}


\For {$2t-1$ values of $\chi_w H_j^\perp$ for which $|c_{j,0}(\chi_w H_j^\perp)|$ is largest}

\State $\overline{\chi_w}:=()$

\For {$1\leq l\leq r-h$}

\State Append $\chi_{w_l}(\Span_{\F_p} \{x_{l,j}\})^\perp$ obtained from the algorithm in \linebreak Lemma \ref{t11} for $\overline{f}(a)=c_{j,ax_{l,j}}(\chi_w H_j^\perp)$, $a\in H$ to $\overline{\chi_w}$ 

\EndFor

\State Reconstruct $\chi_w$ from $\chi_w H_j^\perp$ and $\overline{\chi_w}$

\State $Y(\chi_w):=Y(\chi_w)+1$

\If {$Y(\chi_w)=2t(\lceil r/h\rceil-1)+1$}

\State Append $\chi_w$ to $Z$

\EndIf

\EndFor


\EndFor

\For {$\chi_w\in Z$}

\State $X:=()$

\For {$j$ from $1$ to $n$}

\State Append $c_{j,0}(\chi_w H_j^\perp)$ to $X$

\EndFor

\State $\widehat{f}''(\chi_w):=\Median(\{\Re(x):x\in X\})+i\Median(\{\Im(x):x\in X\})$

\EndFor

\For {$t$ values of $\chi_w$ for which $|\widehat{f}''(\chi_w)|$ is largest}

\State $\widehat{f}'(\chi_w):=\widehat{f}''(\chi_w)$

\EndFor

\noindent \Return $\widehat{f}'$.
\end{algorithmic}
\end{theor}

In the algorithm $\widehat{f}'$, $\widehat{f}''$ and $Y$ are vectors indexed by elements of $\widehat{\F_p^r}$. To see how $\chi_w$ can be reconstructed from $\chi_w H_j^\perp$ and $\overline{\chi_w}$ see the proof of the theorem.

\begin{proof}
Since $c_{j,0}=F_j(f|_{H_j})$ it is enough, from the proof of Theorem \ref{t2}, to prove that that if $|a_{\chi_w}|>2\|\epsilon\|_1/t$ then $\chi_w\in Z$, that is in this case $Y(\chi_w)>2t(m-1)$.

We will first show how $\chi_w$ can be reconstructed from $\chi_w H_j^\perp$ and $\overline{\chi_w}$. We will assume that $H_j=\{(0,\ldots,0,a_{r-h+1},\ldots,a_r)\in \F_p^r\}$ and that $x_{l,j}=(0,\ldots,0,1,0,\ldots,0)$ with $l$-th coefficient 1 and all other coefficients 0 for $1\leq l\leq r-h$ (this can always be assumed, up to changing the basis of $\F_p^r$).

We easily have that
\begin{align*}
H_j^\perp&=\{\chi_{(z_1,\ldots,z_{r-h},0,\ldots,0)}:z_i\in\F_p\},\\
(\Span_{\F_p}\{ x_{l,j}\})^\perp&=\{\chi_{(z_1,\ldots,z_{l-1},0,z_{l+1},\ldots,z_r)}:z_i\in\F_p\}.
\end{align*}
So, using that $\widehat{\F_p^r}/H^\perp\cong \widehat{H}$ for any subgroup $H\leq \F_p^r$,
\begin{align*}
\widehat{H_j}&\cong \widehat{\F_p^r}/H_j^\perp=\{\chi_{(0,\ldots,0,y_{r-h+1},\ldots,y_r)}H_j^\perp:y_i\in\F_p\},\\
\widehat{\Span_{\F_p}\{ x_{l,j}\}}\!&\cong \widehat{\F_p^r}/(\Span_{\F_p}\{ x_{l,j}\})^\perp\!=\!\{\chi_{(0,\ldots,0,y_l,0,\ldots,0)}(\Span_{\F_p}\{ x_{l,j}\})^\perp:y_l\in\F_p\}.
\end{align*}
If $\chi_w H_j^\perp=\chi_{(0,\ldots,0,y_{r-h+1},\ldots,y_r)}H_j^\perp$ and $(\overline{\chi_w})_l=\chi_{(0,\ldots,0,y_l,0,\ldots,0)}(\Span_{\F_p}\{ x_{l,j}\})^\perp$ for $1\leq l\leq r-h$ then $\chi_w=\chi_{(y_1,\ldots,y_r)}$. Notice that in this case
\begin{align*}
(\overline{\chi_w})_1&=\{\chi_z:z_1=y_1&&&&&&\},\\
\vdots&&&\ddots\\
(\overline{\chi_w})_{r-h}&=\{\chi_z:&&&z_{r-h}=y_{r-h}&&&\},\\
\chi_w H_j^\perp&=\{\chi_z:&&&&&z_{r-h+1}=y_{r-h+1},\ldots,z_r=y_r&\},
\end{align*}
In particular $\chi_w$ is the only element contained in $\chi_w H_j^\perp$ and in all of the $(\overline{\chi_w})_l$.

For $\chi_w\in\widehat{\F_p^r}$ we have by Lemma \ref{l8} that if
\[J=\{1\leq j\leq n:\sum_{\chi_z\in\chi_w H_j\setminus\{ \chi_w\}}|a_{\chi_z}|\leq\|\epsilon\|_1/t\}\]
then $|J|>2t(m-1)$. Assume now that $|a_{\chi_w}|>2\|\epsilon\|_1/t$ and $j\in J$. Then $|c_{j,0}(\chi_w H_j^\perp)|>\|\epsilon\|_1/t$ and so it is between the $2t-1$ largest values of $|c_{j,0}(\chi_b H_j^\perp)|$ from Equation \eqref{eq2} from the proof of Theorem \ref{t2}. We will now show that we reconstruct $\chi_w$ from $\chi_w H_j^\perp$ and $\overline{\chi_w}$. This will prove the theorem, since then $Y(\chi_w)>2t(m-1)$.

Clearly $\chi_w\in\chi_w H_j^\perp$. So  it is enough to prove for $1\leq l\leq r-h$ that $\chi_w\in(\overline{\chi_w})_l(\Span_{\F_p}\{ x_{l,j}\})^\perp$. Using Lemma \ref{l7} we have that
\[\overline{f}(a)=c_{j,ax_{l,j}}(\chi_w H_j^\perp)=\sum_{\chi_z\in\chi_w H^\perp}\chi_z(ax_{l,j})a_{\chi_z}\]
As $\Span_{\F_p}\{ x_{l,j}\}\cong \F_p$ we can define $\phi_b(a):=\chi_b(ax_{l,j})$ for $a\in\F_p$ and $\chi_b\in\widehat{\F_p^r}$. Notice that if $\chi_b(\Span_{\F_p}\{ x_{l,j}\})^\perp=\chi_k(\Span_{\F_p}\{ x_{l,j}\})^\perp$ if and only if $\phi_b=\phi_k$. So
\[\overline{f}=\sum_{\chi_b(\Span_{\F_p}\{ x_{l,j}\})^\perp\in \widehat{\F_p^r}/(\Span_{\F_p}\{ x_{l,j}\})^\perp}d_{\phi_b} \phi_b\]
where the coefficients $d_{\phi_b}$ are given by
\[d_{\phi_b}=\sum_{\chi_k\in(\chi_w H_j^\perp)\cap (\chi_b(\Span_{\F_p}\{ x_{l,j}\})^\perp)}a_{\chi_k}.\]

Let
\[c:=\sum_{\chi_w\not=\chi_k\in(\chi_w H_j^\perp)\cap (\chi_w(\Span_{\F_p}\{ x_{l,j}\})^\perp)}|a_{\chi_b}|.\]
Since $j\in J$ we have that $\sum_{\chi_w\not=\chi_b\in\chi_w H_j^\perp}|a_{\chi_b}|\leq \|\epsilon\|_1/t$ (in particular $c=v\|\epsilon\|_1$ for some $0\leq v\leq 1$) and $|a_{\chi_w}|>2\|\epsilon\|_1$. So 
\[|d_{\phi_w}|\geq |a_{\chi_w}|-c>(2-v)\|\epsilon\|_1\]
and
\[\sum_{\phi_b\not=\phi_w}|d_{\phi_b}|\leq 
\sum_{\chi_b\in \chi_w H_j^\perp\setminus\{\chi_w\}}|a_{\chi_b}|-c\leq(1-v)\|\epsilon\|_1.\]
Since $(2-v)\geq 2(1-v)$ for $0\leq v\leq 1$, in this case the algorithm in Lemma \ref{t11} returns $\phi_w$, which corresponds to $\chi_w(\Span_{\F_p}\{ x_{l,j}\})^\perp$ under the isomorphism $\widehat{\F_p^r}/(\Span_{\F_p}\{ x_{l,j}\})^\perp\rightarrow\widehat{\F_p}$ sending $\chi_b(\Span_{\F_p}\{ x_{l,j}\})^\perp\mapsto \phi_b$. So $\chi_w\in(\overline{\chi_w})_l(\Span_{\F_p}\{ x_{l,j}\})^\perp$, which concludes the proof of the theorem.
\end{proof}

\section{Remarks to the extension to groups of the form $(\Z/p^a\Z)^r$}

One can extend the algorithms to work also on groups of the form $(\Z/p^a\Z)^r$ with $a\geq 1$. This can be done as follows:

\begin{enumerate}
 \item
 Construct $m$-generating subspaces $H_1,\ldots,H_n\subseteq\F_p^r$ of dimension $h$ and find a basis $\{z_{1,i},\ldots,z_{h,i}\}$ for each subspace $H_i$.
 
 Assuming coefficients of the vectors $z_{j,i}$ are integers, define $\overline{H_i}:=\langle z_{1,i},\ldots,z_{h,i}\rangle\subseteq (\Z/p^a\Z)^r$ (the subgroup generated by $z_{1,i},\ldots,z_{h,i}$) for $1\leq i\leq n$.
 
 \item
 Extend Lemma \ref{t11} to work for $\Z/p^a\Z$ instead of only for $\F_p$ by taking $k$ maximal with $2^k\geq p^a/3$ and changing $p$ to $p^a$.
 
 \item
 Take $\overline{x_{j,i}}\in\Z/p^a\Z$ with $\langle H_i,\overline{x_{1,i}},\ldots,\overline{x_{r-h,i}}\rangle=(\Z/p^a\Z)^r$.
 
 \item
 Define $\overline{\Gamma_1}$ and $\overline{\Gamma_2}$ similarly to $\Gamma_1$ and $\Gamma_2$.
 
 \item
 In the algorithms substitute $\F_p^r$ with $(\Z/p^a\Z)^r$ and $\widehat{\F_p^r}$ with $\widehat{(\Z/p^a\Z)^r}$.
 \end{enumerate}
 
 It can be proved that the set $\overline{H_i}$ are $m$-generating for $(\Z/p^a\Z)^r$. This depend on square matrices with integer coefficients being singular when reduced to $\Z/p^a\Z$ exactly when the determinant is divisible by $p$, independently of the value of $a$. Also one can prove that
 \[(\Z/p^a\Z)^r\cong \overline{H_i}\times\langle \overline{x_{1,i}}\rangle\times\cdots\times\langle\overline{x_{r-h,i}}\rangle,\]
 which is needed in order to adapt the proof of the theorems. However from $|H_i|=p^{ah}$ and $p^h\geq 4t$ (if $h<r$) we have
  \[|\overline{\Gamma_1}|\geq|H_1|= p^{ah}\geq 4^at^a.\]
 So $|\overline{\Gamma_1}|$ cannot be quadratic in $t$ (unless possibly for $a\leq 2)$. Looking at upper bound on $|\overline{\Gamma_1}|$ we obtain 
 \[|\overline{\Gamma_1}|\leq np^{ah}<16pt^2r^2p^{(a-1)h}.\]
Again as $p^h\geq 4t$ for $h<r$,
\[16pt^2r^2p^{(a-1)h}\geq 16pt^2r^2(4t)^{a-1}=4^{a+1}pt^{a+1}r^2.\]
So, also for $a=2$, the given upper bound is not quadratic in $t$. For this reason we did not extend the paper to groups of the form $(\Z/p^a\Z)^r$.

\section{Further remarks}\label{s3}

From
\[\left|\widehat{\1_H}(\chi_y)\right|=\left|p^{-r}\sum_{x\in H}\chi_y(x)\right|=p^{-r}|H|\1_{H^\perp}(\chi_y)=p^{h-r}\1_{H^\perp}(\chi_y)\]
for any subspace $H$ of $\F_p^r$ of dimension $h$, it can be checked that, if $H_1,\ldots,H_n$ are $m$-generating, then $\sum_i\left|\widehat{\1_{H_i}}(\chi_y)\right|\leq (m-1)p^{h-r}$ for every $1\not=\chi_y\in\widehat{\F_p^r}$ and $\sum_i\left|\widehat{\1_{H_i}}(1)\right|=np^{h-r}:=C$. In particular it can be proved that the matrix with columns labeled by the elements of $\cup_{j\in I}H_j$ (taking the union as a multiset, so that some columns may be repeated), rows labeled by the elements of $\widehat{\F_p^r}$ and coefficients $\chi_y(g)/\sqrt{C}$ satisfies the RIP property of rank $t$ and constant $(t-1)(m-1)/n$. In particular if $n>2(t-1)(m-1)$ and the subspaces $H_i$ are $m$-generating then we could approximate any $t$-sparse Fourier signal also using reconstruction algorithms based on the RIP property.

Some slight variation of the algorithms presented in Theorem \ref{t2} and \ref{t9'} (based on similar algorithms from \cite{MR2628828}) can also be constructed taking $n$ to be $2t(m-1)$, which however returns functions with
\[\|(\widehat{f}+\epsilon)-\widehat{f}'\|_1\leq (1+2t)\|\epsilon\|_1.\]
If $\|\epsilon\|_1>0$ these algorithms give worse approximations than the ones described here. However if $\|\epsilon\|_1=0$ then also these algorithms reconstruct $f$ and they work on smaller (about half the size) sets. Since sizes of sampling sets and complexities of the algorithms are about the same (they differ only by a small multiplicative constant), these algorithms have not been presented here.

\section{Acknowledgments}

The author thanks Hartmut F\"uhr for help in reviewing the paper.

The results presented here are part of the author Ph.d. thesis \cite{lm}. The results have also already been presented in the form of an extended abstract in \cite{lm2}.

The author was supported during her Ph.d. by the DFG grant GRK 1632 for the Graduiertenkolleg Experimentelle und konstruktive Algebra at RWTH Aachen University.


\begin{thebibliography}{20}
 
\bibitem{b1} J. Bailey, M. A. Iwen, C. V. Spencer, On the design of deterministic matrices for fast recovery of Fourier compressible functions, {\em SIAM J. Matrix Anal. Appl.} 33, 263-289, 2012.

\bibitem{MR2817651} J. Bourgain, S. Dilworth, K. Ford, S. Konyagin, D. Kutzarova, Explicit constructions of RIP matrices and related problems, {\em Duke Math. J.} 159, 145-185, 2011.

\bibitem{MR2230846} E. J. Cand{\`e}s, J. K. Romberg, T. Tao, Stable signal recovery from incomplete and inaccurate measurements, {\em Comm. Pure Appl. Math.} 59, 1207-1223, 2006.

\bibitem{MR2243152} E. J. Candes, T. Tao, Decoding by linear programming, {\em IEEE Trans. Inform. Theory} 51, 4203-4215, 2005.

\bibitem{MR2371999} R. A. DeVore, Deterministic constructions of compressed sensing matrices, {\em J. Complexity} 23, 918-925, 2007.

\bibitem{MR3100033} S. Foucart, H. Rauhut, {\em A mathematical introduction to compressive sensing}, Applied and Numerical Harmonic Analysis, Birkh\"auser/Springer, New York, 2013.

\bibitem{5464880} J. Haupt, L. Applebaum, R. Nowak, On the restricted isometry of deterministically subsampled Fourier matrices, In {\em Information Sciences and Systems (CISS), 2010 44th Annual Conference on}, 1-6, 2010.

\bibitem{MR2628828} M. A. Iwen, Combinatorial sublinear-time Fourier algorithms, {\em Found. Comput. Math.} 10, 303-338, 2010.
 
\bibitem{ms} F. J. Macwilliams, N. J. A. Sloane, {\em The theory of error-correcting codes}, North-Holland Publishing Co., Amsterdam-New York-Oxford, 1977, North-Holland Mathematical Library, Vol. 16.

\bibitem{lm} L. Morotti, {\em Explicit construction of universal sampling sets for finite abelian and symmetric groups}, Ph.d. thesis, RWTH Aachen University, 2014.
   
\bibitem{lm2} L. Morotti, Reconstruction of Fourier sparse signals over elementary abelian groups, {\em ENDM} 43, 161-167, 2013.
 
\bibitem{MR2417886} M. Rudelson, R. Vershynin, On sparse reconstruction from Fourier and Gaussian measurements, {\em Comm. Pure Appl. Math.} 61, 1025-1045, 2008.

\bibitem{MR3164973} A. M. Tillmann, M. E. Pfetsch, The computational complexity of the restricted isometry property, the nullspace property, and related concepts in compressed sensing, {\em IEEE Trans. Inform. Theory} 60, 1248-1259, 2014.

\end{thebibliography}
\end{document}